\documentclass[12pt]{article}

\usepackage{amsmath,amsthm,amsfonts,amssymb,bbm}
\usepackage{dsfont,extarrows,enumerate,url}
\usepackage{fullpage}
\usepackage{xcolor}

\DeclareMathAlphabet\mathbfcal{OMS}{cmsy}{b}{n}

\numberwithin{equation}{section}

\theoremstyle{plain} \newtheorem{theorem}{Theorem}[section]
\theoremstyle{plain} \newtheorem{proposition}[theorem]{Proposition}
\theoremstyle{plain} \newtheorem{lemma}[theorem]{Lemma}
\theoremstyle{plain} \newtheorem{corollary}[theorem]{Corollary}
\theoremstyle{definition} \newtheorem{definition}[theorem]{Definition}
\theoremstyle{definition} 
\theoremstyle{remark} \newtheorem{remark}[theorem]{Remark}
\theoremstyle{remark} \newtheorem{example}[theorem]{Example}

\renewcommand{\P}{\mathbf P}

\newcommand{\E}{\mathbf E}
\newcommand{\R}{\mathbb R}
\newcommand{\Rd}[1][d]{\R^{#1}}

\newcommand{\NN}{\mathbb N}

\newcommand{\salg}{\mathfrak{F}}
\newcommand{\fC}{\mathfrak{C}}
\newcommand{\ssalg}{\mathfrak{A}}

\newcommand{\sA}{\mathcal{A}}

\newcommand{\sF}{\mathcal F}
\newcommand{\sK}{\mathcal K}
\newcommand{\sM}{\mathcal M}
\newcommand{\sN}{\mathcal N}
\newcommand{\sP}{\mathcal P}

\newcommand{\one}{\mathbf{1}}

\newcommand{\risk}{\mathsf{r}}
\newcommand{\ve}{\mathsf{e}}
\newcommand{\gae}{\mathsf{g}}

\newcommand{\vE}{\mathbfcal{E}}
\newcommand{\gE}{\mathbfcal{G}}

\newcommand{\Lp}[1][p]{\mathsf{L}^{#1}}

\newcommand{\dint}{\mathrm{d}}
\newcommand{\emm}{\hat{\mu}_n}

\DeclareMathOperator{\conv}{conv}

\DeclareMathOperator{\cl}{cl}
\DeclareMathOperator{\esssup}{ess\,sup}

\DeclareMathOperator{\diam}{\mathrm{diam}}
\DeclareMathOperator{\Met}{\mathrm{M}}

\renewcommand{\phi}{\varphi}
\renewcommand{\epsilon}{\varepsilon}
\newcommand{\eps}{\varepsilon}
\newcommand{\Prob}[1]{\P\left\{ #1 \right\}}

\renewcommand{\subset}{\subseteq}
\renewcommand{\supset}{\supseteq}

\usepackage{fancybox}
\newlength{\querylen}
\begin{document}

\title{Convex bodies generated by sublinear expectations of random
  vectors\footnote{Supported by the Swiss National Science Foundation
    Grant No. 200021\_175584}}

\author{Ilya Molchanov \&\ Riccardo Turin\\
\normalsize Institute of Mathematical Statistics and Actuarial Science,
University of Bern\\
\normalsize ilya.molchanov@stat.unibe.ch, riccardo.turin@stat.unibe.ch}
\date{\today}

\maketitle




\begin{abstract}
  We show that many well-known transforms in convex geometry (in
  particular, centroid body, convex floating body, and Ulam floating
  body) are special instances of a general construction, relying on
  applying sublinear expectations to random vectors in Euclidean
  space. We identify the dual representation of such convex bodies and
  describe a construction that serves as a building block for all so
  defined convex bodies.  Sublinear expectations are studied in
  mathematical finance within the theory of risk measures.  In this
  way, tools from mathematical finance yield a whole variety of new
  geometric constructions.

  Keywords: centroid body; convex floating body; metronoid; quantile;
  risk measure; sublinear expectation; Ulam floating body

  MSC 2020: Primary: 60D05, Secondary: 52A27, 62G30, 91G70
\end{abstract}



\section{Introduction}
\label{sec:introduction}

The concept of sublinear expectation is essential in
mathematical finance, where it is used to quantify the operational
risk, see \cite{del12,foel:sch04}. The sublinearity property reflects
the financial paradigm, saying that the diversification decreases the
risk, and so the risk of a diversified portfolio is dominated by the sum
of the risks of its components.  Sublinear expectations are closely
related to solutions of backward stochastic differential equations,
see \cite{pen19}.

A sublinear expectation $\ve$ is a sublinear (positively homogeneous
and convex) map from the space $\Lp(\R)$ (or another linear space of
random variables) to $(-\infty,\infty]$, and so may be regarded as a
convex function on an infinite-dimensional space, see \cite{zal02} for
a thorough account of convex analysis tools in the
infinite-dimensional setting.

In this paper we use a sublinear expectation $\ve$ to associate with
each $p$-integrable random vector $\xi$ in $\R^d$ a convex closed set
$\vE_\ve(\xi)$ in $\R^d$. This is done by letting the support function
of $\vE_\ve(\xi)$ be the sublinear expectation $\ve$ applied to the
scalar product $\langle\xi,u\rangle$. For instance, if $\ve$ is the
$\Lp$-norm and $\xi$ is symmetric, $\vE_\ve(\xi)$ becomes the centroid
body associated to the distribution of $\xi$ as introduced by Petty
\cite{pet61} for $p=1$ and Lutwak and Zhang \cite{lut:zhan97} for a
general $p$. If $\ve$ is the average quantile of
$\langle\xi,u\rangle$, one obtains convex closed sets called
metronoids and studied by Huang and Slomka \cite{huan:slom18}. Further
examples are given by expected random polytopes, which also form a
special case of our construction.

We commence with Section~\ref{sec:subl-expect}, giving the definition
of sublinear expectation of random variables, explaining their dual
representation and presenting several examples. We mention the
particularly important Kusuoka representation which expresses any
law-determined sublinear expectation in terms of integrated quantiles
and describe a novel construction (called the maximum extension)
suitable to produce parametric families of sublinear expectations from
each given one.

Section~\ref{sec:constr-conv-bodi-2} presents our construction of
convex closed sets $\vE_\ve(\xi)$ generated by a random vector $\xi$
and a given sublinear expectation $\ve$.
Section~\ref{sec:star-shaped-sets} describes a generalisation based on
relaxing some properties of the underlying numerical sublinear
expectations, namely, replacing them with gauge
functions. This construction yields centroid bodies \cite{lut:zhan97}
and half-space depth-trimmed regions \cite{nag:sch:wer18}, the latter
are closely related to convex floating bodies introduced in
\cite{schut:wer90} and their weighted variant from \cite{MR3841844}.

One of the most important sublinear expectations is based on using
weighted integrals of the quantile function. The corresponding convex
bodies are studied in Section~\ref{sec:conv-bodi-gener}, where we show
their close connection to metronoids \cite{huan:slom18} and
zonoid-trimmed regions \cite{kos:mos97}. The Kusuoka representation of
numerical sublinear expectations yields
Theorem~\ref{th:representation}, which provides a representation of a
general convex set $\vE_\ve(\xi)$ (derived from $\xi$ using a
sublinear expectation $\ve$) in terms of Aumann integrals of
metronoids.
We further provide a uniqueness result for the distribution of $\xi$ on
the basis of a family of convex bodies generated by it, and also a
concentration result for random convex sets constructed from the
empirical distribution of $\xi$.

Section~\ref{sec:conv-bodi-gener-1} specialises our general
construction to the case when $\xi$ is uniformly distributed on a
convex body $K$ (that is, a compact convex set in $\R^d$ with nonempty
interior), and so $\vE_\ve(\xi)$ yields a transform
$K\mapsto \vE_\ve(K)=\vE_\ve(\xi)$. We derive several properties of
this transformation for general $\ve$, in particular, establish the
continuity of such maps 
in the Hausdorff metric.

In special cases, our construction yields $\Lp$-centroid bodies (see
\cite{lut:zhan97} and \cite[Sec.~10.8]{schn2}) and Ulam floating
bodies recently introduced in \cite{huan:slom:wer18}.  The latter form
a particularly important special setting, which is confirmed by
showing that all transformations $K\mapsto\vE_\ve(K)$ can be
expressed in terms of Ulam floating bodies. For instance,
Corollary~\ref{cor:centroid} provides a representation of the centroid
body of an origin symmetric $K$ as the convex hull of dilated Ulam
floating bodies of $K$.  In this course, results for sublinear
expectations yield a new insight into the well-known aforementioned
constructions of convex bodies, deliver some new relations between
them, and provide a general source of nonlinear transformations of
convex bodies.  Finally, we formulate several conjectures.

\section{Sublinear expectations of random variables}
\label{sec:subl-expect}

\subsection{Definition and dual representation}
\label{sec:definition}

Let $(\Omega,\salg,\P)$ be a nonatomic probability space, and let
$\Lp(\R^d)$ denote the family of all $p$-integrable random vectors in
$\R^d$, with $p\in[1,\infty]$. Endow $\Lp(\R^d)$ with the
$\sigma(\Lp,\Lp[q])$-topology, which is the weak-star topology based
on the pairing of $\Lp(\R^d)$ and $\Lp[q](\R^d)$ with
$p^{-1}+q^{-1}=1$, see \cite[Sec.~5.14]{alip:bor06}. Denote
$\R_+=[0,\infty)$.

The following definition amends the standard definition of sublinear
expectations of random variables (see, e.g., \cite{pen19}) by
including the extra lower semicontinuity property, which is often
additionally imposed.

\begin{definition}
  \label{def:subl}
  A \textit{sublinear expectation} is a function
  $\ve:\Lp(\R)\to(-\infty,\infty]$ with $p\in[1,\infty]$,
  satisfying the following properties for all $\beta,\beta'\in\Lp(\R)$:
  \begin{enumerate}[i)]
  \item monotonicity: $\ve(\beta)\leq\ve(\beta')$ if $\beta\leq\beta'$
    a.s.;
  \item translation equivariance: $\ve(\beta+a)=\ve(\beta) + a$ for all
    $a\in\R$, and $\ve(0)=0$;
  \item positive homogeneity: $\ve(c\beta)=c\ve(\beta)$ for all $c>0$;
  \item subadditivity: $\ve(\beta+\beta')\leq\ve(\beta)+\ve(\beta')$,
  \item lower semicontinuity in $\sigma(\Lp,\Lp[q])$, that is,
    \begin{equation*}
      \ve(\beta)\leq \liminf_{n\to\infty} \ve(\beta_n)
    \end{equation*}
    for each sequence $\{\beta_n,n\geq1\}$ converging to $\beta$ in
    the weak-star topology $\sigma(\Lp,\Lp[q])$.
  \end{enumerate}
\end{definition}

The sublinear expectation $\ve$ is often referred to as
\emph{numerical} one, in contrast with the set-valued expectation
introduced in Section~\ref{sec:constr-conv-bodi-2}.
The translation equivariance property implies that $\ve(a)=a$ for each
deterministic $a$. The sublinear expectation $\ve$ is said to be
\emph{finite} if it takes a finite value on all $\beta\in\Lp(\R)$.  

\begin{example}[Relation to coherent risk measures]
  For $\beta\in\Lp(\R)$, define $\risk(\beta)=\ve(-\beta)$. The
  obtained antimonotonic and subadditive function is called a
  \emph{coherent risk measure} of $\beta$, see \cite{del12} and
  \cite[Def.~4.5]{foel:sch04}. The negative of the risk is said to be
  a utility function, see \cite{del12}.

  A random variable $\beta$ is said to be acceptable if its risk is at
  most zero.  If $\beta$ is the financial position at the terminal
  time, its risk $\risk(\beta)$ yields the smallest amount $a$ of
  capital which should be reserved at the initial time to render
  $\beta+a$ acceptable; this amount may be negative if
  $\risk(\beta)<0$, and then capital can be released or
  invested.
  The subadditivity property of the risk
  (equivalently, of $\ve$) is the manifestation of the financial
  principle, saying that diversification decreases the risk. Many
  results from the theory of risk measures can be easily reformulated
  for sublinear expectations. For instance, from the theory of risk
  measures, it is known that the lower semicontinuity property always
  holds if $p\in[1,\infty)$ and $\ve$ takes only finite values, see
  \cite{kain:rues09}. 
\end{example}

While the following result is well known for risk measures
\cite[Cor.~4.18]{foel:sch04} and sublinear expectations
\cite[Th.~1.2.1]{pen19}, we provide its proof for completeness.

\begin{theorem}
  \label{thr:dual-numeric}
  A functional $\ve:\Lp(\R)\to(-\infty,\infty]$ is a sublinear
  expectation if and only if 
  \begin{equation}
    \label{eq:6}
    \ve(\beta)=\sup_{\gamma\in\sM_\ve,\E\gamma=1} \E(\gamma\beta),
  \end{equation}
  where $\sM_\ve$ is a convex $\sigma(\Lp[q],\Lp)$-closed cone in
  $\Lp[q](\R_+)$.
\end{theorem}
\begin{proof}
  \textsl{Sufficiency} is easy to confirm by a direct check of the
  properties.
  
  \textsl{Necessity.} Let $\sA$ be the family of $\beta\in\Lp(\R)$,
  such that $\ve(\beta)\leq0$. The sublinearity property yields that
  $\sA$ is a convex cone. The lower semicontinuity property implies
  that this cone is weak-star closed. The polar cone to $\sA$ is
  defined as
  \begin{equation}
    \label{eq:9}
    \sA^o=\{\gamma\in\Lp[q](\R):\E(\gamma\beta)\leq 0\; \text{for all
    }\; \beta\in\sA\}. 
  \end{equation}
  Since $-\one_A\in\sA$ for the indicator of any event $A$, all random
  variables from $\sA^o$ are a.s.\ nonnegative.  The bipolar theorem
  from functional analysis (see, e.g., \cite[Th.~5.103]{alip:bor06})
  yields that $(\sA^o)^o=\sA$. Hence,
  \begin{align*}
    \ve(\beta)&=\inf\{a\in\R:(\beta-a)\in\sA\}\\
    &=\inf\{a\in\R:\E((\beta-a)\gamma)\leq 0\;\text{for all }\;
      \gamma\in\sA^o\}\\
    &=\inf\{a\in\R:\E(\gamma\beta)\leq a\E(\gamma)\;\text{for all }\;
      \gamma\in\sA^o\}.
  \end{align*}
  Thus, \eqref{eq:6} holds with $\sM_\ve=\sA^o$. 
\end{proof}

Representation \eqref{eq:6} is called the \emph{dual} representation
of $\ve$.  It is easy to see that each $\gamma$ in \eqref{eq:6} can be
chosen to be a function of $\beta$, namely, the conditional
expectations $\E(\gamma|\beta)$.

A sublinear expectation is said to be \emph{law-determined} (often
named law invariant) if it attains the same value on identically
distributed random variables, and this is the case for all examples
considered in this paper. In terms of the representation \eqref{eq:6},
this means that, for each $\gamma\in\sM_\ve$, the set $\sM_\ve$
contains all random variables sharing the same distribution with
$\gamma$.

A sublinear expectation is said to be \emph{continuous from below} if
it is continuous on all almost surely convergent increasing sequences
of random variables in $\Lp(\R)$.  It follows from \cite{kain:rues09}
that each finite sublinear expectation on $\Lp(\R)$ with
$p\in[1,\infty)$ is continuous from below. Every law-determined
continuous from below sublinear expectation on a nonatomic probability
space is \emph{dilatation monotonic}, meaning that
\begin{equation}
  \label{eq:8}
  \ve(\E(\beta|\ssalg))\leq \ve(\beta)
\end{equation}
for each sub-$\sigma$-algebra $\ssalg$ of $\salg$, see
\cite[Cor.~4.59]{foel:sch04}. In particular, $\E\beta\leq \ve(\beta)$
for all $\beta\in\Lp(\R)$.

\subsection{Average quantiles and the Kusuoka representation}
\label{sec:aver-quant-thek}

For a fixed value of $\alpha\in(0,1]$ and $\beta\in\Lp[1](\R)$, define
\begin{equation}
  \label{eq:3}
  \ve_\alpha(\beta)=\frac{1}{\alpha}
  \int_{1-\alpha}^{1}q_t(\beta)\dint t, 
\end{equation}
where 
\begin{equation}
  \label{eq:quantile}
  q_t(\beta)=\sup\{s\in\R : \Prob{\beta\leq s}< t\}
  =\inf\{s\in\R : \Prob{\beta\leq s}\geq t\}
\end{equation}
is the $t$-quantile of $\beta$. Because of integration, the choice of
a particular quantile in case of multiplicities is immaterial.  This
sublinear expectation is subsequently called the \emph{average
  quantile}. In particular, $\ve_1(\beta)=\E\beta$ is the mean. If
$\beta$ has a nonatomic distribution, then
$\ve_\alpha(\beta)=\E(\beta|\beta\geq q_{1-\alpha}(\beta))$.

The value of $\risk(\beta)=\ve_\alpha(-\beta)$ is obtained by
averaging the quantiles of $\beta$ at levels between 0 and
$\alpha$. This risk measure is well studied in finance and widely
applied in practice under the name of the average Value-at-Risk or
expected shortfall, see, e.g., \cite{acer:tas02}. By computing the
dual cone at \eqref{eq:9} or rephrasing the representation of the risk
measure $\ve_\alpha(-\beta)$ from \cite[Th.~4.1]{kain:rues09}, one can
derive the following dual representation
\begin{equation}
  \label{eq:4}
  \ve_{\alpha}(\beta)=\sup_{\gamma\in\Lp[\infty]([0,\alpha^{-1}]),\E\gamma=1} 
  \E(\gamma\beta).
\end{equation}
This immediately yields that the average quantiles satisfy all
properties imposed in Definition~\ref{def:subl}.
\medskip

Average quantiles form a building block for all other law-determined
sublinear expectations. The following result for risk measures is
known as the \emph{Kusuoka representation}: it was first obtained by
Kusuoka \cite{kusuoka2001} in case $p=\infty$ and can also be found in
\cite[Cor.~4.58]{foel:sch04} and \cite[Th.~32]{del12}; the
$\Lp$-variant follows from the Orlicz space version proved in
\cite{MR3778360}. For its validity, it is essential that the
probability space is nonatomic.

\begin{theorem}
  \label{thr:law-inv-rep}
  Each 
  law-determined sublinear expectation on $\Lp(\R)$
  with $p\in[1,\infty]$ can be represented as
  \begin{equation}
    \label{eq:12}
    \ve(\beta)=\sup_{\nu\in\sP_\ve}
    \int_{(0,1]}\ve_\alpha(\beta)\nu(\dint\alpha),
  \end{equation}
  where $\sP_\ve$ is the family of probability measures $\nu$ on
  $(0,1]$ such that
  $\int_{(0,1]}\ve_\alpha(\beta)\nu(\dint\alpha)\leq 0$ whenever
  $\ve(\beta)\leq0$.
\end{theorem}

It is possible to show that $\ve$ is finite on $\Lp(\R)$ if and only
if the function $t\mapsto \int_{(t,1]}s^{-1}\nu(\dint s)$ is
$q$-integrable on $(0,1]$ with respect to the Lebesgue measure for all
$\nu\in\sP_\ve$.
If $\ve$ is finite and $p\in[1,\infty)$, one can provide a
constructive representation of $\sP_\ve$ in terms of the extremal
points of the set $\sM^1_\ve=\{\gamma\in\sM_\ve: \E\gamma=1\}$, where
$\sM_e$ is defined in \eqref{eq:6}. The case $p=\infty$ requires extra
arguments, since a norm bounded set in $\Lp[1]$ is not necessarily
weakly compact, hence, the supremum in \eqref{eq:6} is not necessarily
attained.  Since $(\Omega,\sF,\P)$ is nonatomic, we can assume without
loss of generality that $\Omega$ is the interval $[0,1]$ equipped with
its Borel $\sigma$-algebra and the Lebesgue measure $\P$. Let
$\gamma:[0,1)\to [0,\infty)$ be a nondecreasing right-continuous
function that is extremal in $\sM^1_\ve$.  Define the probability
measure $\nu_\gamma$ on $(0,1]$ by letting
\begin{displaymath}
  \nu_\gamma((0,\alpha))=\int_{[1-\alpha,1)}(\gamma(t)-\gamma(1-\alpha))\dint t.
\end{displaymath}
and $\nu(\{1\})=\gamma(0)$. 
It is shown in \cite{shapiro2013} that $\sP_\ve$ can be chosen to be
the set of $\nu_\gamma$ for the family of all right-continuous
nondecreasing functions $\gamma$ which are extremal in $\sM^1_\ve$.

\subsection{Examples of sublinear expectations}
\label{sec:exampl-subl-expect}

A simple example of a sublinear expectation is provided by the
essential supremum
\begin{displaymath}
  \ve(\beta)=\esssup \beta, 
\end{displaymath}
which is finite for all $\beta\in\Lp[\infty](\R)$.  If
$\alpha\downarrow0$, then the average quantile $\ve_\alpha(\beta)$
increases to the (possibly, infinite) value $\ve_0(\beta)$, which is
equal to the essential supremum of $\beta$.  Next, we discuss more
involved constructions of sublinear expectations.

\begin{example}[Spectral sublinear expectation]
  \label{ex:spectral}
  Let $\phi:(0,1]\to\R_+$ be a nonincreasing function such that
  $\int_0^1\phi(t)\dint t=1$, $\phi$ is called a spectral function. Then 
  \begin{equation}
  \label{eq:spec}
    \ve_{\int\!\!\phi}(\beta)=\int_{(0,1]} q_{1-t}(\beta)\phi(t)\dint t
  \end{equation}
  is called a spectral sublinear expectation, see \cite{acer02} for
  the closely related definition of the spectral risk measure. By
  Fubini's theorem, $\ve_{\int\!\!\phi}$ admits the following
  equivalent representation
  \begin{equation}
  \label{eq:spec_aver}
    \ve_{\int\!\!\phi}(\beta)
    =\int_{(0,1]} \ve_\alpha(\beta)\nu(\dint\alpha),
  \end{equation}
  where $\ve_\alpha$ is given by \eqref{eq:3} and $\nu$ is the
  probability measure on $(0,1]$ with
  \begin{equation}
    \label{eq:13}
    \phi(t)=\int_{(t,1]}s^{-1}\nu(\dint s), \quad t\in(0,1].
  \end{equation}
  Conversely, for any probability measure $\nu$ on $(0,1]$,
  \eqref{eq:spec_aver} yields a spectral sublinear expectation.
  The set $\sP_\ve$ in the Kusuoka representation of
  $\ve_{\int\!\!\phi}(\beta)$ consists of the single probability
  measure $\nu$, so the right-hand side \eqref{eq:12} is the supremum
  over a family of spectral sublinear expectations. 
\end{example}

\begin{example}[One-sided moments]
  \label{ex:fischer}
  The $\Lp$-norm $\|\beta\|_p$ satisfies all properties of a sublinear
  expectation but the monotonicity and translation equivariance. It is
  possible to come up with a norm-based sublinear expectation on
  $\Lp(\R)$ with $p\in[1,\infty)$ by letting
  \begin{equation}
    \label{eq:norm}
    \ve_{p,a}(\beta)=\E\beta+ a\big(\E(\beta-\E \beta)_+^p\big)^{1/p}
  \end{equation}
  with $a\in[0,1]$, where $x_+=\max(x,0)$ denotes the positive part of
  $x\in\R$. The corresponding risk measure was introduced in
  \cite{fis03}. Note that $\ve_{p,a}(\beta)=\frac{a}{2}\|\beta\|_p$ if
  $\beta$ is symmetric. Translation equivariance and positive
  homogeneity of $\ve_{p,a}$ are obvious.  The subadditivity of the
  second term follows from $(t+s)_+\leq (t)_++(s)_+$ and the
  subadditivity of the $\Lp$-norm. To prove the monotonicity, we first
  observe that since $\ve_{p,a}$ is subadditive, we only need to show
  that $\ve_{p,a}(\gamma)\leq 0$ for any almost surely negative integrable
  $\gamma$. Indeed, substituting $(\gamma-\E\gamma)_+\leq -\E\gamma$
  in \eqref{eq:norm} implies that
  $\ve_{p,a}(\gamma)\leq\E\gamma-a\E\gamma\leq 0$.
  
  The sublinear expectation given by \eqref{eq:norm} admits the dual
  representation \eqref{eq:6} with the cone $\sM_\ve$ generated by the
  family of random variables $\gamma=1+a(\zeta-\E\zeta)$ for all
  $\zeta\in\Lp[q](\R_+)$ with $\|\zeta\|_q\leq 1$, see
  \cite[p.~46]{del12}.
  The family $\sP_\ve$ from \eqref{eq:12} is
  explicitly known only for $p=1$; it consists of probability measures
  obtained as $(1-at)\delta_1+at\delta_t$, which is the weighted sum
  of the Dirac measures at $1$ and $t$ for $t\in[0,1]$. Then
  \begin{equation}
    \label{eq:11}
    \ve_{1,a}(\beta)=\sup_{t\in[0,1]}
    \Big[(1-at)\E\beta+at\ve_t(\beta)\Big]
    =\E\beta + a\sup_{t\in[0,1]}t\ve_t(\beta-\E\beta). 
  \end{equation}
  Recall in this relation that
  \begin{displaymath}
    t\ve_t(\beta)=\int_{1-t}^1 q_s(\beta)\dint s,
  \end{displaymath}
  so that the supremum on the right-hand side of \eqref{eq:11} is
  indeed the expectation of $(\beta-\E\beta)_+$. 
\end{example}

\begin{example}[Expectile]
  \label{ex:expectile}
  Following \cite{bel:klar:mul:14}, define the \emph{expectile}
  $\ve_{[\tau]}(\beta)$ of a random variable $\beta\in\Lp[1](\R)$ at level
  $\tau\in (0,1)$ as the (necessarily, unique) solution  $x\in\R$ of
  \begin{displaymath}
    \tau\E (\beta-x)_+=(1-\tau)\E(x-\beta)_+.
  \end{displaymath}
  If $\tau\in[1/2,1)$, then the expectile is a sublinear expectation,
  see \cite{bel:klar:mul:14}. For $\tau=1/2$, we obtain the mean of
  $\beta$. For $\tau\in[1/2,1)$, the dual representation holds with
  $\sM_\ve$ being the set of $\gamma\in\Lp[\infty](\R_+)$ such that
  the ratio between the essential supremum and the essential infimum of
  $\gamma$ is at most $\tau/(1-\tau)$. The Kusuoka representation
  holds with 
  \begin{displaymath}
    \ve_{[\tau]}(\beta)
    =\sup_{t\in[0,2-1/\tau]}
    \Big[(1-t)\E\beta+t\ve_{\frac{(1-\tau)t}{(2\tau-1)(1-t)}}(\beta)\Big]. 
  \end{displaymath}
\end{example}

\subsection{Maximum extension}
\label{sec:maximum-extension}

Let $\ve$ be a law-determined sublinear expectation on $\Lp(\R)$ with
$p\in[1,\infty]$.  The following construction suggests a way of
extending $\ve$ to a monotone parametric family of sublinear
expectations.  For a fixed $m\geq1$, define
\begin{equation}
  \label{eq:max-gen}
  \ve^{\vee m}(\beta)=\ve(\max(\beta_1,\dots,\beta_m)),
\end{equation}
where $\beta_1,\dots,\beta_m$ are independent copies of
$\beta\in\Lp(\R)$. All properties in Definition~\ref{def:subl} are
straightforward and we refer to this sublinear expectation as the
\emph{maximum extension} of $\ve$. Let us stress that this extension
applies only to law-determined sublinear expectations. 

It is possible to obtain a family of such expectations
$\ve^{\vee(\lambda)}$ continuously parametrised by
$\lambda\in(0,1]$. For this, $m$ is replaced by a geometrically
distributed random variable $N$ with parameter $\lambda$, that is,
$\Prob{N=k}=(1-\lambda)^{k-1}\lambda$, $k\geq1$. Define
\begin{displaymath}
  \ve^{\vee(\lambda)}=\ve(\max(\beta_1,\dots,\beta_N)), \quad \lambda\in(0,1].
\end{displaymath}
This family of sublinear expectations interpolates between
$\ve^{\vee (1)}(\beta)=\ve(\beta)$ and $\ve^{\vee (0)}(\beta)$ which
is set to be $\esssup\beta$.

\begin{example}
  \label{ex:max-ext}
  The maximum extension can be applied to the average quantile risk
  measure $\ve_\alpha$; the result is denoted by $\ve^{\vee
    m}_\alpha$.  For  $\alpha=1$, we obtain the \emph{expected maximum}
  \begin{equation}
    \label{eq:max}
    \ve^{\vee m}_1(\beta)=\E\max(\beta_1,\dots,\beta_m).
  \end{equation}
  Note that
  \begin{align*}
    \ve^{\vee m}_1(\beta)=\int_0^1 q_t
    (\max\{\beta_1,\dots,\beta_m\})\dint t
    =\int_0^1 q_{t^\frac{1}{m}}(\beta) \dint t
    =m\int_0^1 t^{m-1} q_t(\beta)\dint t\,.
  \end{align*}
  For $m\geq2$, $\ve^{\vee m}_1$ is the spectral sublinear expectation
  given at \eqref{eq:spec} with $\phi(t)=m(1-t)^{m-1}$, equivalently,
  \eqref{eq:spec_aver} with
  $\nu(\dint t)=m(m-1)t(1-t)^{m-2}\dint t$.
  Similar calculations yield that
  \begin{multline}
    \label{eq:alpha-m}
    \ve_\alpha^{\vee m}(\beta)=\ve_\alpha(\max(\beta_1,\dots,\beta_m))
    =\frac{m(m-1)}{\alpha}
    \int _0^{1-(1-\alpha)^{1/m}} t(1-t)^{m-2}\ve_t(\beta)\dint t\\
    +\frac{m}{\alpha}(1-\alpha)^{(m-1)/m}(1-(1-\alpha)^{1/m})
    \ve_{1-(1-\alpha)^{1/m}}(\beta).
  \end{multline}
\end{example}

\section{Measure-generated convex sets}
\label{sec:constr-conv-bodi-2}

Fix a law-determined sublinear expectation $\ve$ on $\Lp(\R)$,
$p\in[1,\infty]$.  For a $p$-integrable probability measure $\mu$ on
$\R^d$, equivalently, for a random vector $\xi\in\Lp(\R^d)$ with
distribution $\mu$, define
\begin{equation}
  \label{eq:support-function}
  h(u)=\ve(\langle\xi,u \rangle), \quad u\in \R^d,
\end{equation}
where $\langle\xi,u\rangle$ denotes the scalar product in $\R^d$.  The
function $h$ is subadditive
\begin{displaymath}
  h(u+u')=\ve(\langle\xi ,u+u'\rangle)
  \leq\ve(\langle\xi ,u\rangle) + \ve(\langle\xi ,u'\rangle)
  =h(u)+h(u'),
\end{displaymath}
and homogeneous
\begin{displaymath}
  h(cu)=\ve(\langle\xi ,cu \rangle)
  =c\ve(\langle\xi ,u \rangle)=ch(u),\quad  c\geq 0.
\end{displaymath}
Furthermore, $h$ is lower semicontinuous, since $\langle\xi ,u_n
\rangle\to\langle\xi ,u \rangle$ in $\sigma(\Lp,\Lp[q])$ if $u_n\to u$
as $n\to\infty$ and $\ve$ is assumed to be lower semicontinuous.
These three properties identify support functions of convex closed
sets, see \cite[Th.~1.7.1]{schn2}.  Therefore, there exists a
(possibly, unbounded) convex closed set $F$ such that its support function
\begin{displaymath}
  h(F,u)=\sup\{\langle x,u\rangle:\; x\in F\}
\end{displaymath}
is given by \eqref{eq:support-function}.  This set is denoted by
$\vE_\ve(\xi)$ or $\vE_\ve(\mu)$. The construction can be summarised by the
equality
\begin{equation}
  \label{eq:10}
  h(\vE_\ve(\xi),u)=\ve(\langle\xi ,u\rangle), \quad u\in\R^d. 
\end{equation}

The following result shows that $\vE_\ve(\xi)$ is a set-valued sublinear
function of $\xi$, called a \emph{set-valued sublinear expectation}
generated by $\ve$.
In other instances, we pass to $\vE_\ve$ the sub- and superscripts of
$\ve$, e.g., $\vE_{[\tau]}$ is obtained by choosing $\ve$ to be the
expectile $\ve_{[\tau]}$.

For convex closed sets $F,F'$, their (closed) Minkowski sum $F+F'$ is
the closure of $\{x+x':\; x\in F,x'\in F'\}$, and the dilation of $F$
by $c>0$ is $cF=\{cx:\; x\in F\}$.

\begin{theorem}
  \label{thr:subl-vE}
  Fix $p\in[1,\infty]$ and a law-determined sublinear expectation
  $\ve$ defined on $\Lp(\R)$.  The corresponding map $\vE_\ve$ (given at
  \eqref{eq:10}) from $\Lp(\R^d)$ to the family of convex closed sets
  in $\R^d$ satisfies the following properties:
  \begin{enumerate}[i)]
  \item monotonicity: if $\xi\in F$ a.s. for a convex closed $F$, then
    $\vE_\ve(\xi)\subset F$;
  \item singleton preserving: $\vE_\ve(a)=\{a\}$ for all deterministic
    $a$;
  \item affine equivariance $\vE_\ve(A\xi+a)=A\vE_\ve(\xi)+a$ for all matrices
    $A$ and $a\in\R^d$;
  \item subadditivity: $\vE_\ve(\xi+\eta)\subset \vE_\ve(\xi)+\vE_\ve(\eta)$;
  \item lower semicontinuity of support functions, that is,
    $h(\vE_\ve(\xi),u)\leq \liminf_{n\to\infty} h(\vE_\ve(\xi_n),u)$ for all
    $u\in\R^d$ if $\xi_n\to\xi$ in $\sigma(\Lp,\Lp[q])$;
  \item if $\ve(\beta)$ is finite for all $\beta\in\Lp(\R)$, then the
    map $\xi\mapsto\vE_\ve(\xi)$ is continuous in the Hausdorff metric
    (see, \cite[Sec,~1.8]{schn2}) with respect to the norm on $\Lp$;
  \item if $\ve$ is continuous from below, then $\vE_\ve(\xi)$ contains
    the expectation $\E\xi$.
  \end{enumerate}
\end{theorem}
\begin{proof}
  Property (i) holds since $\langle\xi,u\rangle\leq h(F,u)$ and in
  view of the monotonicity property of $\ve$.
  Property (ii) directly follows from the construction, and, for the
  affine equivariance, note that
  \begin{displaymath}
    h(\vE_\ve(A\xi+a),u)=\ve(\langle \xi,A^\top u\rangle)
    +\langle a,u\rangle
    =h(\vE_\ve(\xi),A^\top u)+\langle a,u\rangle
    =h(A\vE_\ve(\xi)+a,u).
  \end{displaymath}
  The subadditivity follows from
  \begin{displaymath}
    h(\vE_\ve(\xi+\eta),u)
    =\ve(\langle\xi+\eta,u\rangle)
    \leq \ve(\langle \xi,u\rangle)+\ve(\langle\eta,u\rangle)
    =h(\vE_\ve(\xi),u)+h(\vE_\ve(\eta),u). 
  \end{displaymath}
  
  If $\xi_n\to\xi$ in $\sigma(\Lp(\R^d),\Lp[q](\R^d))$, then
  $\langle\xi_n,u\rangle\to\langle\xi,u\rangle$ in
  $\sigma(\Lp(\R),\Lp[q](\R))$. By the lower semicontinuity of $\ve$,
  \begin{displaymath}
    \ve(\langle\xi,u\rangle)\leq \liminf_{n\to\infty}
    \ve(\langle\xi_n,u\rangle). 
  \end{displaymath}
  This implies the lower semicontinuity of the support functions. 

  Property (vi) follows from the Extended Namioka Theorem, which says
  that every finite sublinear expectation is continuous with respect
  to the norm topology, see \cite{bia:fri09}. Recall that sublinear
  expectations on $\Lp[\infty]$ is also Lipschitz. Hence,
  $\ve(\langle\xi_n,u\rangle)\to \ve(\langle\xi,u\rangle)$ if
  $\xi_n\to\xi$ in $\Lp$. The convergence of support functions implies
  the convergence of the corresponding sets in the Hausdorff metric,
  see \cite[Th.~1.8.15]{schn2}.
  
  Finally, (vii) is a consequence of the dilatation monotonicity
  property \eqref{eq:8}.
\end{proof}

\begin{example}
  If $\ve$ is the essential supremum, then $\vE_\ve(\xi)$ equals the
  closed convex hull of the support of $\xi$.
\end{example}

If $p=\infty$, then an easy argument shows that the map
$\xi\mapsto\vE_\ve(\xi)$ between $\Lp[\infty](\R^d)$ and the family of
convex compact sets in $\R^d$ is 1-Lipschitz, that is, the Hausdorff
distance between $\vE_\ve(\xi)$ and $\vE_\ve(\eta)$ is at most
$\|\xi-\eta\|_\infty$ for all $\xi,\eta\in\Lp[\infty](\R^d)$. Indeed,
\begin{align*}
  h(\vE_\ve(\xi),u)-h(\vE_\ve(\eta),u)
  &=\ve(\langle\xi,u\rangle)-\ve(\langle\eta,u\rangle)\\
  &\leq \ve(\langle\eta,u\rangle + \|\xi-\eta\|_\infty)
  -\ve(\langle\eta,u\rangle)=\|\xi-\eta\|_\infty
\end{align*}
for all unit $u\in\R^d$. 

If $\xi,\eta\in\Lp(\R^d)$ and $\E(\eta|\xi)=0$ a.s., then the
dilatation monotonicity property~\eqref{eq:8} implies that
\begin{displaymath}
  \vE_\ve(\xi+\eta)\supset \vE_\ve(\E(\xi+\eta|\xi))=\vE_\ve(\xi). 
\end{displaymath}
Hence, if $\xi_1,\xi_2,\ldots$ is a sequence of i.i.d.\ centred
$p$-integrable random vectors, then $\vE_\ve(\xi_1+\cdots+\xi_n)$,
$n\geq1$, is a growing sequence of nested convex sets in $\R^d$.

\begin{remark}
  \label{rem:order}
  If $\xi$ is dominated by $\eta$ in the \emph{convex order}, meaning
  that $\E f(\xi)\leq \E f(\eta)$ for all convex functions $f$, then
  $\vE_\ve(\xi)\subset \vE_\ve(\eta)$, see \cite[Cor.~4.59]{foel:sch04}. In
  particular, the sequence $\vE_\ve(\xi_n)$, $n\geq1$, grows if
  $(\xi_n)_{n\geq0}$ is a martingale.
\end{remark}

\begin{example}
  Let $\langle \xi,u\rangle$ be distributed as $\zeta\|u\|_L$, where
  $\zeta$ is a random variable and $\|\cdot\|_L$ is a certain norm on
  $\R^d$ with $L$ being the unit ball; then $\xi$ is called
  pseudo-isotropic, see, e.g., \cite{ger:kean:mis00}. In this case,
  $\vE_\ve(\xi)=cL^o$, where
  \begin{equation}
    \label{eq:polar}
    L^o=\{u:\; h(L,u)\leq1\}
  \end{equation}
  is the \emph{polar set} to $L$ and
  $c=\ve(\zeta)=\ve(\langle \xi,u\rangle)$ for any given
  $u\in\partial L$. For instance, this is the case if $\xi$ is
  symmetric $\alpha$-stable with $\alpha\in(1,2]$; then $\vE_\ve(\xi)$
  is expressed in terms of the associated convex body of $\xi$, see
  \cite{mo09}. If $\xi$ is Gaussian, then $L^o$ is the ellipsoid
  determined by the covariance matrix of $\xi$ and translated by the
  mean of $\xi$.
\end{example}

The dual representation of $\ve$ given by
Theorem~\ref{thr:dual-numeric} immediately implies the following
result. 

\begin{corollary}
  \label{cor:dual-E}
  The set-valued sublinear expectation generated by $\ve$ can be
  represented as
  \begin{equation}
    \label{eq:13a}
    \vE_\ve(\xi)=\cl\{\E(\xi\gamma):\; \gamma\in\sM_\ve,\E\gamma=1\},
  \end{equation}
  where $\cl$ denotes the topological closure in $\R^d$ and $\sM_\ve$
  is the family of probability measures from \eqref{eq:6}.
\end{corollary}

The convexity of $\sM_\ve$ implies that the set on the right-hand side of
\eqref{eq:13a} is convex. This set can be written as the intersection
of the cone $\{(\E\gamma,\E(\xi\gamma)):\; \gamma\in\sM_\ve\}$ with the
set $\{1\}\times\Rd$ and then projected on its last $d$-components.

\begin{remark}
  It is possible to construct a variant of the set $\vE_\ve(\xi)$ by
  applying the underlying sublinear expectation $\ve$ to the positive
  part $(\langle\xi,u\rangle)_+$ of the scalar product of $\xi$ and
  $u$. The obtained function is the support function of a convex
  closed set, which may be considered a sublinear expectation of the
  segment $[0,\xi]$, see \cite{mol:mueh19} for a study of sublinear
  expectations with set-valued arguments. 
\end{remark}

\section{Convex gauges}
\label{sec:star-shaped-sets}

We sometimes consider a variant of the sublinear expectation which is
a positive homogeneous, subadditive and lower semicontinuous function
$\gae:\Lp(\R)\to(-\infty,\infty]$ and so is not necessarily monotone
or translation equivariant. We refer to this function as a
\emph{convex gauge}. The most important example is the $\Lp$-norm, so
that $\gae(\beta)=\|\beta\|_p$, which is convex but not translation
equivariant.

For a lower semicontinuous convex gauge $\gae$, we define $\gE(\xi)$
as the convex closed set such that
\begin{displaymath}
  h(\gE(\xi),u)=\gae(\langle\xi,u\rangle), \quad u\in\R^d.
\end{displaymath}
It is easily seen that $\gae(\langle\xi,u\rangle)$ is indeed a support
function.

\begin{example}
  \label{ex:centroid}
  Let $\gae(\beta)=\|\beta\|_p$. For $\xi\in\Lp(\R^d)$, the convex
  body $\gE(\xi)$ is the $\Lp$-centroid of $\xi$ (or of its
  distribution $\mu$). These convex bodies have been introduced in
  \cite{pet61} for $p=1$ and in \cite{lut:zhan97} for a general $p$,
  and further thoroughly studied, see, e.g.,
  \cite{MR2349721,hab:sch09,MR2880241}.
\end{example}


In some cases, $\gae$ fails to be convex. For instance, this is the
case for $\Lp$-norm with $p\in(0,1)$.
Another important case arises when $\gae(\beta)$ is the quantile
function $q_t(\beta)$ given by \eqref{eq:quantile} for a fixed
$t\in(0,1)$, which is known to be not necessarily subadditive in
$\beta$. In the absence of subadditivity, it is natural to consider
the largest convex set whose support function is dominated by the
quantile function of $\langle x,u\rangle$, namely, let
\begin{equation}
  \label{eq:char-depth}
  D_\delta(\xi)=\bigcap_{u\in\Rd} 
  \left\{x\in\Rd:\; \langle x,u\rangle\le 
    q_{1-\delta}(\langle\xi ,u\rangle)\right\}.
\end{equation}
The set $D_\delta(\xi)$ is called the \emph{depth-trimmed region} of
$\xi$. The support function of $D_\delta(\xi)$ may be strictly less than
$q_{1-\delta}(\langle\xi ,u\rangle)$, for example, if $\xi$ is
uniformly distributed on a triangle on the plane, see \cite{lei86}.
The set $D_\delta(\xi)$ is necessarily empty if $\xi$ is nonatomic and
$\delta\in(1/2,1]$.

The set $D_\delta(\xi)$ is related to the \emph{Tukey (or half-space)
  depth} (see \cite{tuk75}), which associates to a point $x$ the
smallest $\mu$-content of a half-space containing $x$, where $\mu$ is
the distribution of $\xi$.  The depth-trimmed region of $\xi$ is the
excursion set of the Tukey depth, so that
\begin{equation}
  \label{eq:5}
  D_{\delta}(\xi)=\bigcap_{\mu(H)>1-\delta} H\,,
\end{equation}
where $H$ runs through the collection of all closed half-spaces.  If
$\xi$ has \emph{contiguous support} (that is, the support of
$\langle\xi ,u\rangle$ is connected for every $u$), then
\eqref{eq:char-depth} holds with $q$ being any other quantile function
in case of multiplicities, and the intersection in \eqref{eq:5} can be
taken over half-spaces $H$ with $\mu(H)\geq 1-\delta$, see \cite{bru19,kon:miz12}.

\begin{example}
  \label{ex:conv-floating}
  Let $\xi$ be uniformly distributed on a convex body $K$. Then
  $D_\delta(\xi)$ is the \emph{convex floating body} of $K$, see
  \cite{schut:wer90} and \cite{MR1962608}. A variant of this concept
  for nonuniform distributions on $K$ has been studied in \cite{MR3841844}.
\end{example}

Recall that a random vector $\xi$ with distribution $\mu$ is said to
have \emph{$k$-concave distribution}, with $k\in[-\infty,\infty]$, if
\begin{displaymath}
  \mu(\theta A+ (1-\theta)B)\ge
  \begin{cases}
    \min\{\mu(A),\mu(B)\}   &\text{if }k=-\infty,     \\
    \mu(A)^\theta\mu(B)^{(1-\theta)}&\text{if }k=0,\\
    (\theta\mu(A)^k+(1-\theta)\mu(B)^k)^{1/k} 	&\text{otherwise},
  \end{cases}
\end{displaymath}
for all Borel sets $A$ and $B$ and $\theta\in[0,1]$. In case of
$k=0$, the measure $\mu$ is called \emph{log-concave}.  The next
theorem establishes some conditions under which
$q_\delta(\langle\xi,u\rangle)$ is a support function; it is a direct
consequence of \cite[Th.~6.1]{bob10}.

\begin{theorem}
  \label{th:bob}
  Let $\xi$ be a symmetric $k$-concave random vector with $k\ge -1$
  and such that the support of $\xi$ is full-dimensional. Then
  \begin{displaymath}
    h(D_\delta(\xi),u)= q_{1-\delta}(\langle\xi ,u\rangle),\quad u\in\R^d,
  \end{displaymath}
  for all $\delta\in (0,1/2)$.
\end{theorem}

\section{Convex bodies generated by average quantiles}
\label{sec:conv-bodi-gener}

\subsection{Metronoids and zonoid-trimmed regions}
\label{sec:metr-zono-rand}

For $\xi\in\Lp[1](\R^d)$ and $\alpha\in(0,1]$, denote by
$\vE_\alpha(\xi)$ the convex set generated by the average quantile
sublinear expectation $\ve_\alpha$ given by \eqref{eq:3}. Such convex
sets are hereafter called \emph{average quantile sets}. In
particular, $\vE_1(\xi)=\E\xi$. Since $\ve_\alpha$ is finite on
$\Lp[1](\R)$, the set $\vE_\alpha(\xi)$ is compact. Noticing that
$q_t(-\beta)=-q_{1-t}(\beta)$, it is easy to see that
$\vE_\alpha(\xi)$ has nonempty interior for all $\alpha\in(0,1)$,
hence, is a convex body. The set $\vE_\alpha(\xi)$ increases as
$\alpha$ decreases to zero with limit $\vE_0(\xi)$, being the
convex hull of the support of $\xi$.

The following result relates average quantile sets and the
\emph{zonoid-trimmed regions} introduced in \cite{kos:mos97} as
\begin{displaymath}
  Z_\alpha(\xi)=\big\{\E(\xi f(\xi)):\; 
  f:\Rd\to[0,\alpha^{-1}]\,\text{measurable and}\,\E f(\xi)=1\big\}.
\end{displaymath}

\begin{proposition}
  \label{prop:zon-trim-reg}
  For all $\alpha\in(0,1]$, $\vE_\alpha(\xi)=Z_\alpha(\xi)$.
\end{proposition}
\begin{proof}
  Representation \eqref{eq:4} yields that
  \begin{displaymath}
    h(\vE_\alpha(\xi),u)=\ve_\alpha(\langle\xi,u \rangle)
    =\sup_{\gamma\in\Lp[\infty]([0,\alpha^{-1}]),\,\E\gamma=1}
    \langle \E(\gamma\xi),u\rangle\,.
  \end{displaymath}
  Noticing that any $\gamma$ in the last expression can be replaced
  with $\E(\gamma|\xi)$ yields that 
  \begin{displaymath}
    \sup_{\substack{\gamma\in\Lp[\infty]([0,\alpha^{-1}])\\ \E\gamma=1}}
    \langle \E(\gamma\xi),u\rangle
    =\sup_{\substack {f:\;\Rd\to[0,\alpha^{-1}]\\ \E f(\xi)=1} }
    \langle\E(\xi f(\xi)),u\rangle. \qedhere
  \end{displaymath}
\end{proof}

Let $\mu$ be a locally finite Borel measure on $\R^d$.  Denote by
$\Lp[1]_\mu([0,1])$ the family of functions $f:\R^d\to[0,1]$ such that
$\int xf(x)\mu(\dint x)$ exists.  The set
\begin{displaymath}
  \Met(\mu)=\left\{\int_{\R^d} xf(x)\mu(\dint x):\;
  f\in\Lp[1]_\mu([0,1]),\int_{\R^d}f \dint\mu=1\right\}
\end{displaymath}
has the support function 
\begin{displaymath}
  h(\Met(\mu),u)=\sup_{f\in\Lp[1]_\mu([0,1]),\int f\dint\mu=1}
  \int \langle x,u\rangle f(x)\mu(\dint x), \quad u\in\R^d. 
\end{displaymath}
The set $\Met(\mu)$ was introduced in \cite{huan:slom18} and called the
\emph{metronoid} of $\mu$. This definition applies also for possibly
infinite measures $\mu$, e.g.,  if $\mu$ is the Lebesgue measure, then
$\Met(\mu)=\R^d$, since each point $x\in\Rd$ can be obtained by letting
$f$ be the indicator of the unit ball centred at $x$ normalised by the
volume of the unit ball. Furthermore, $\Met(\mu)$ is empty if the total
mass of $\mu$ is less than one, and $\Met(\mu)$ is the singleton
$\int x\mu(\dint x)$ if $\mu$ is an integrable probability measure.
The following result establishes a relation between metronoids and
average quantile sets.

\begin{proposition}
  \label{prop:metro-averquant}
  Let $\mu$ be an integrable probability measure on $\Rd$.  Then
  $\Met(\alpha^{-1}\mu)=\vE_\alpha(\mu)$ for any $\alpha\in(0,1]$.
\end{proposition}
\begin{proof}
  Consider a random vector $\xi$ with distribution $\mu$.  By
  \eqref{eq:4}, for every $u\in\Rd$ the support function of
  $\Met(\alpha^{-1}\mu)$ is
  \begin{align*}
    h(\Met(\alpha^{-1}\mu),u)
    &=\sup_{0\leq f\leq 1,\;\int f \dint\mu =\alpha}
    \int \langle x,u\rangle f(x) \alpha^{-1}\mu(\dint x)\\
    &= \sup_{0\leq f\leq \alpha^{-1},\;\E f(\xi)=1}
    \E \big(\langle \xi,u\rangle f(\xi)\big)\\
    &=\sup_{\gamma\in\Lp[\infty]([0,\alpha^{-1}]),\,\E\gamma=1}
    \E\left(\langle \xi,u\rangle \gamma\right)\\
    &=\ve_\alpha(\langle \xi,u\rangle)
    =h(\vE_\alpha(\xi),u)\,,
  \end{align*}
  where in the second equality $f\alpha^{-1}$ was replaced by $f$ and
  later $f(\xi)$ by $\gamma$. 
\end{proof}

\begin{example}
  Let $\xi$ have a discrete distribution with atoms at $x_1,\dots,x_n$
  of probabilities $p_1,\dots,p_n$.  Then $\vE_\alpha(\xi)$ is the
  polytope
  \begin{displaymath}
    \vE_\alpha(\xi)=\left\{\sum_{i=1}^n\lambda_i p_i x_i:\; 
      \lambda_1,\dots,\lambda_n\in[0,\alpha^{-1}],\ 
      \sum_{i=1}^{n}\lambda_i p_i=1\right\}\,,
  \end{displaymath}
  see \cite[Prop.~2.3]{huan:slom18}, where this is proved for
  metronoids.
\end{example}

\subsection{A representation of general $\vE_\ve(\xi)$}
\label{sec:spectr-conv-bodi}

Fix $\xi\in\Lp[1](\R^d)$ and consider the average quantile sets
$\vE_\alpha(\xi)$ as a set-valued function of $\alpha\in(0,1]$. Let
$\nu$ be a probability measure on $(0,1]$, which appears in the
spectral sublinear expectation \eqref{eq:spec_aver} from
Example~\ref{ex:spectral}. The closed \emph{Aumann integral} (see
\cite{aum65}) of the set-valued function
$\alpha\mapsto\vE_\alpha(\xi)$ is the convex closed set
$\vE_{\int\!\!\phi}(\xi)$, whose support function at any direction $u$
equals the integral of the support function, so that
\begin{equation}
  \label{eq:hint}
  h(\vE_{\int\!\!\phi}(\xi),u)=\int_{(0,1]}
  h(\vE_\alpha(\xi),u)\nu(\dint \alpha), \quad u\in\R^d. 
\end{equation}
Recognising the right-hand side as
$\ve_{\int\!\!\phi}(\langle\xi,u\rangle)$, it is immediately seen that
$\vE_{\int\!\!\phi}(\xi)$ is the set-valued sublinear expectation
generated by the spectral numerical one from
Example~\ref{ex:spectral}. Equivalently, $\vE_{\int\!\!\phi}(\xi)$
equals the closure of the set of integrals of all measurable
integrable functions $f(\alpha)$, $\alpha\in(0,1]$, such that
$f(\alpha)\in \vE_\alpha(\xi)$ for all $\alpha$, see \cite{aum65} and
\cite[Sec.~2.1.2]{mo1}. This is reflected by writing
\begin{equation}
  \label{eq:2}
  \vE_{\int\!\!\phi}(\xi)=\cl \int_{(0,1]}\vE_\alpha(\xi)\nu(\dint\alpha).
\end{equation}
Since $\vE_\alpha(\xi)$ increases to the closed convex hull of the
support of $\xi$ as $\alpha\downarrow0$, the set
$\vE_{\int\!\!\phi}(\xi)$ is not necessarily bounded.

The following result provides a representation of the set
$\vE_\ve(\xi)$ constructed using a general law-determined sublinear
expectation $\ve$. It confirms that the average quantile sets
(equivalently, metronoids) are building blocks for a general
$\vE_\ve(\xi)$. Denote by $\conv A$ the closed convex hull of a set $A$ in
$\R^d$.

\begin{theorem}
  \label{th:representation}
  For each $\xi\in\Lp(\R^d)$ and a set-valued sublinear expectation
  $\vE_\ve(\xi)$ generated by a law-determined sublinear expectation
  $\ve$, we have
  \begin{displaymath}
    \vE_\ve(\xi)=\conv\bigcup_{\nu\in\sP_\ve} \int_{(0,1]}\vE_\alpha(\xi)\nu(\dint\alpha),
  \end{displaymath}
  where $\sP_\ve$ is the family probability measures $\nu$ on $(0,1]$
  from the Kusuoka representation of $\ve$, see \eqref{eq:12}.
\end{theorem}
\begin{proof}
  By Theorem~\ref{thr:law-inv-rep},
  \begin{align*}
    \ve(\langle\xi,u\rangle)=\sup_{\nu\in\sP_\ve} \int_{(0,1]}
                              \ve_\alpha(\langle\xi,u\rangle)\nu(\dint\alpha)
    =\sup_{\nu\in\sP_\ve} \int_{(0,1]} h(\vE_\alpha(\xi),u)\nu(\dint\alpha).
  \end{align*}
  The proof is completed by noticing \eqref{eq:hint}, using the
  notation \eqref{eq:2} and the fact that the supremum of support
  functions is the support function of the closed
  convex hull of the involved sets.
\end{proof}

\subsection{Average quantile sets as integrated depth-trimmed regions}
\label{sec:relation-half-space}

Under the symmetry and log-concavity assumptions on $\xi$, the average
quantile sets $\vE_\alpha(\xi)$ can be characterised as set-valued
integrals of the depth-trimmed regions (equivalently, weighted
floating bodies) $D_\delta(\xi)$ introduced in
\eqref{eq:char-depth}. Similarly to \eqref{eq:hint}, the closed Aumann
integral of the function $t\mapsto D_t(\xi)$ with respect to a measure
$\nu$ on $[0,1]$ is defined as the convex set whose support function
equals the integral of the support functions of $D_t(\xi)$, that is,
\begin{displaymath}
  h\Big(\int_0^1 D_t(\xi)\nu(\dint t),u\Big)=\int_0^1 h(D_t(\xi),u)\nu(\dint t),
  \quad u\in\R^d. 
\end{displaymath}
If the measure $\nu$ attaches positive mass to the set of $t\in[0,1]$
where $D_t(\xi)$ is empty, the integral is set to be the empty set.

The following result establishes relationships between average
quantile sets (or metronoids) and depth-trimmed regions. Its second
part generalises \cite[Th.~1.1]{huan:slom:wer18}, which concerns the
case of $\xi$ supported by a convex body.

\begin{theorem}
  \label{th:inclusion}
  Let $\xi\in\Lp[1](\R^d)$. Then
  \begin{equation}
    \label{eq:14}
    D_\alpha(\xi)\subseteq
    \frac{1}{\alpha}\int_0^\alpha D_t(\xi)\dint t
    \subseteq \vE_\alpha(\xi).
  \end{equation}
  If $\xi$ has a log-concave distribution, then
  \begin{equation}
    \label{eq:stat2}
    D_{\frac{e-1}{e}\alpha}(\xi) \subset\vE_\alpha(\xi)\subset D_{\frac{\alpha}{e}}(\xi)
  \end{equation}
  for every $\alpha\in(0,1]$.
\end{theorem}
\begin{proof}
  By definition of the average quantile set,
  \begin{align*}
    h(\vE_\alpha(\xi),u)=\frac{1}{\alpha}\int_{1-\alpha}^1
    q_t(\langle\xi,u\rangle) \dint t
    \geq \frac{1}{\alpha}\int_{1-\alpha}^1
    h(D_{1-t}(\xi),u) \dint t
    =\frac{1}{\alpha}\int_0^\alpha
    h(D_t(\xi),u) \dint t,
  \end{align*}
  where the inequality follows from \eqref{eq:char-depth}. Finally,
  \eqref{eq:14} follows from the monotonicity of $D_t(\xi)$. 
  
  Fix $u\in\R^d$. Consider $\beta=\langle\xi ,u\rangle$ and note that
  the distribution $\nu$ of $\beta$ is log-concave by the invariance
  of the log-concavity property under projection.  For
  \eqref{eq:stat2}, it suffices to show that
  \begin{equation}
    \label{eq:thesis}
    q_{\left(1-\frac{e-1}{e}\alpha\right)}(\beta)
    \le	\ve_\alpha(\beta)
    \le	q_{\left(1-\frac{1}{e}\alpha\right)}(\beta). 
  \end{equation}
  Being the projection of a log-concave vector, $\beta$ is either
  deterministic or absolutely continuous with connected support. In
  the first case \eqref{eq:thesis} becomes trivial, thus we can assume
  that $\beta$ is absolutely continuous with connected support. In
  particular, $q$ in \eqref{eq:thesis} can be equivalently chosen to
  be the left- or the right-quantile function.  Observe that, for
  measurable sets $A$ and $B$, convex $C$ and $\theta\in[0,1]$,
  \begin{align*}
    \nu\left(A\cap C\right)^{\theta}\nu\left(B\cap C\right)^{1-\theta}
    &\le \nu\left(\theta\left(A\cap C\right)
      +\left(1-\theta\right)\left(B\cap C\right)\right) \\
    &= \nu\left(\left(\theta A\cap\theta C\right)
      +\left(\left(1-\theta\right)B\cap
        \left(1-\theta\right)C\right)\right)\\
    &\le \nu\left(\left(\theta A+\left(1-\theta\right)B\right)
      \cap\left(\theta C+\left(1-\theta\right)C\right)\right)\\
    &=\nu\left(\left(\theta A+\left(1-\theta\right)B\right)\cap C\right)\,.
  \end{align*}
  Therefore, the probability measure obtained by restricting $\nu$ to
  the interval $(q_{1-\alpha}(\beta),\infty)$ and normalising by the
  factor $\alpha{-1}$ is log-concave, and we consider a random
  variable $X$ with such distribution.  It follows from the theory of
  risk measures (see, e.g., \cite[Prop.~2.1]{tadese20:_relat}), that
  for the case of absolutely continuous random variables, supremum in
  the characterisation of $\ve_\alpha(\beta)$ in \eqref{eq:4} is
  attained at $\gamma=\alpha^{-1}\one_{\{\beta>q_{1-\alpha}(\beta)\}}$,
  which implies
  \begin{equation}
    \label{eq:exp=avequant}
    \E X=\alpha^{-1}\E\left(\beta\one_{\{\beta>q_{1-\alpha}(\beta)\}}\right)
    =\ve_\alpha(\beta)\,.
  \end{equation}
  It follows from \cite[Eq.~(5.7)]{bob10} that for any
  log-concave random variable $X$,
  \begin{equation}
    \label{eq:ineq}
    e^{-1}\le\Prob{X>\E X}\le 1-e^{-1}\,.
  \end{equation}
  Therefore, \eqref{eq:exp=avequant} and \eqref{eq:ineq} yield that
  \begin{displaymath}
    e^{-1}\le\alpha^{-1}\nu(\ve_\alpha(\beta),\infty)
    \le 1-e^{-1}\,.
  \end{displaymath}
  Hence,
  \begin{displaymath}
    e^{-1}\alpha \le \Prob{\beta > \ve_\alpha(\beta)}
    \le \left(1-e^{-1}\right)\alpha\,,
  \end{displaymath}
  which implies \eqref{eq:thesis}, given that $\beta$ has connected
  support.
\end{proof}

\subsection{A uniqueness result for maximum extensions}
\label{sec:uniq-result-from}

A single set $\vE_\ve(\xi)$ surely does not characterise the distribution
of $\xi$. However, families of such sets can be sufficient to recover
the distribution of $\xi$. 

\begin{example}
  \label{ex:avq-xi}
  Assume that $\xi,\eta\in\Lp[1](\R^d)$ and consider the average
  quantile sets $\vE_\alpha(\xi)$ and $\vE_\alpha(\eta)$.  If
  $\vE_\alpha(\xi)=\vE_\alpha(\eta)$ for all $\alpha\in(0,1/2]$, then
  $\xi$ and $\eta$ have the same distribution. This follows from
  Proposition~\ref{prop:zon-trim-reg} and \cite[Th.~5.6]{kos:mos97}.
\end{example}

Since the definition of $\vE_\ve(\xi)$ is based on the univariate
sublinear expectation $\ve$ applied to the projections of $\xi$, the
following result is a straightforward application of the
Cram\'er--Wold theorem, see, e.g., \cite[Cor.~5.5]{MR1876169}.

\begin{proposition}
  \label{prop:uniq}
  A family of sets $\vE_\ve(\xi)$, $\ve\in E$, generated by sublinear
  expectations $\ve$ from a certain family $E$ uniquely identifies the
  distribution of $\xi\in\Lp(\R^d)$ if and only if the family of the
  underlying univariate sublinear expectations $\ve(\beta)$,
  $\ve\in E$, uniquely identifies the distribution of any
  $\beta\in\Lp(\R)$.
\end{proposition}

Natural families of sublinear expectations arise by applying the
maximum extension to a given sublinear expectation. 

\begin{example}
  \label{ex:max-unique}
  Consider the expected maximum sublinear expectation $\ve^{\vee m}_1$
  given by \eqref{eq:max}. Then the convex body $\vE^{\vee m}_1(\xi)$
  is the expectation $\E P_{m}$ of the random polytope $P_{m}$
  obtained as the convex hull of $m$ independent copies of $\xi$, see
  \cite[Sec.~2.1]{mo1}.  It is well known that the sequence $\ve^{\vee
    m}_1(\beta)$, $m\geq1$, uniquely identifies the distribution of
  $\beta\in\Lp[1](\R)$, see \cite{hoe53} and \cite{gal:kot78}. As a
  consequence, the nested sequence $\E P_m$, $m\geq 1$, of convex
  bodies uniquely determines the distribution of $\xi$, see
  \cite{vi87}.
\end{example}

Applying the maximum extension \eqref{eq:max-gen} to the spectral
sublinear expectation $\ve_{\int\!\!\phi}(\cdot)$ yields the
sublinear expectation $\ve_{\int\!\!\phi}^{\vee m}(\cdot)$ and the
corresponding sequence of nested convex bodies
$\vE_{\int\!\!\phi}^{\vee m}(\xi)$, $m\geq1$.

\begin{theorem}
  Let $\xi,\eta\in\Lp[1](\R^d)$.  For any constant $c\geq 0$, consider
  the spectral function $\phi(t)=(c+1)(1-t)^c$.  If
  \begin{displaymath}
    \vE_{\int\!\!\phi}^{\vee m}(\xi)
    =\vE_{\int\!\!\phi}^{\vee m}(\eta)\,,\quad m\ge 1\,,
  \end{displaymath}
  then $\xi$ and $\eta$ have the same distribution.
\end{theorem}
\begin{proof}
  In view of Proposition~\ref{prop:uniq}, it suffices to prove this
  result for two random variables $\beta$ and $\gamma$.  For any
  integer $m\ge 1$, we have
  \begin{displaymath}
    \int_0^1 q_{1-t}\left(\max(\beta_1,\dots,\beta_m)\right)\phi(t)\dint t
    =\int_0^1 q_{1-t}\left(\max(\gamma_1,\dots,\gamma_m)\right)\phi(t)\dint t\,,
  \end{displaymath}
  where $\beta_i,\gamma_i$, $i=1,\dots,m$, are independent copies of
  $\beta,\gamma$, respectively.  By a change of variables,
  \begin{align*}
    \int_0^1 q_{1-t}\left(\max(\beta_1,\dots,\beta_m)\right)\phi(t)\dint t
    &=(c +1)\int_0^1 q_{t}\left(\max(\beta_1,\dots,\beta_m)\right)t^c\dint t\\
    &=(c +1)\int_0^1 q_{t^\frac{1}{m}}(\beta)t^c\dint t\\
    &=m(c +1)\int_0^1 q_{s}(\beta)s^{c m + m - 1}\dint s\,.
  \end{align*}
  Therefore,
  \begin{equation*}
    \int_0^1 f(s)s^{(c+1)(m-1)}\dint s=0\,,\quad m\ge 1\,,
  \end{equation*}
  with
  \begin{equation*}
    f(s)=s^c\left(q_{s}(\beta)-q_{s}(\gamma)\right)\in\Lp[1]([0,1])\,.
  \end{equation*}
  The family
  \begin{equation*}
    \mathcal{A}=\left\{c_0+\sum_{i=1}^n c_ix^{(c+1)m_i}:\;
      n,m_1,\dots,m_n\in\NN,\,c_0,\dots,c_n\in\R\right\}
  \end{equation*}
  is an algebra of continuous functions separating the points on
  $[0,1]$.  By linearity of the Lebesgue integral
  \begin{equation*}
    \int_0^1 f(s)a(s)\dint s=0
  \end{equation*}
  for all $a\in\mathcal{A}$. The Stone--Weierstrass theorem (see,
  e.g., \cite[Th.~4.45]{folland99}) yields that
  \begin{equation*}
    \int_0^1 f(s)g(s)\dint s=0
  \end{equation*}
  for all continuous functions $g$ on $[0,1]$.  Therefore, $f$
  vanishes almost everywhere, so the proof is complete.
\end{proof}

\subsection{Concentration of empirical average quantile sets}
\label{sec:empirical}

Let $\xi\in\Lp(\R^d)$ with distribution $\mu$. Consider the empirical
random measure constructed by $n$ independent copies
$\xi_1,\dots,\xi_n$ of $\xi$ as
\begin{equation}
  \label{eq:emm}
  \emm=\frac{1}{n}\sum_{i=1}^n \delta_{\xi_i}\,,\quad n\ge 1\,,
\end{equation}
where $\delta_x$ is the one point mass measure at $x\in\Rd$.  The
average quantile convex body $\vE_\alpha(\emm)$ generated by $\emm$
is a random convex set, which approximates the body $\vE_\alpha(\mu)$
as $n$ grows to infinity.  In fact, the sequence
$\{\vE_\alpha(\emm),\,n\ge 1\}$ almost surely converges to
$\vE_\alpha(\mu)$ in the Hausdorff metric, as directly follows from
\cite[Th.~5.2]{kos:mos97} and Proposition~\ref{prop:zon-trim-reg}.
The following theorem provides probabilistic bounds for this
convergence.

\begin{theorem}
  \label{th:emp}
  Let $\mu$ be a probability measure with bounded support of diameter
  $R$, and let $r$ be the largest radius of a centred Euclidean ball
  contained in the average quantile set $\vE_\alpha(\mu)$ for some
  $\alpha\in(0,1)$. For all $\eps>0$ and $n\in\NN$,
  \begin{displaymath}
    \P\Big\{(1-\eps)\vE_\alpha(\mu)\subset\vE_\alpha(\emm)
      \subset (1+\eps)\vE_\alpha(\mu)\Big\}
    \geq 1-6^{d+1}(1+1/\eps)^d
    \exp\left\{-\frac{\alpha\eps^2 r^2 n}{44R^2}\right\}.
  \end{displaymath}
\end{theorem}

We use the following auxiliary result.
  
\begin{lemma}[see \protect{\cite[Lemma 5.2]{fres:vit14}}]
  \label{lem:net}
  Let $K$ be a convex body which contains the origin in its
  interior. For each $\delta\in(0,1/2)$, there exists a set
  $\sN\subset\partial K$ with cardinality at most $(3/\delta)^d$ such
  that each $v\in\partial K$ satisfies
  \begin{displaymath}
    v=w_0+\sum_{i=1}^\infty \delta_i w_i
  \end{displaymath}
  for $w_i\in\sN$, $i\geq0$, and $\delta_i\in [0,\delta^i]$, $i\geq1$.
\end{lemma}

\begin{proof}[Proof of Theorem \ref{th:emp}]
  On a (possibly enlarged) probability space $\Omega\times\Omega'$,
  let $\xi$ be a $\mu$-distributed random vector and let $\hat\xi_n$
  take one of the values $\xi_1,\dots,\xi_n$ with equal
  probabilities. For any fixed $u\in\Rd$,
  \begin{displaymath}
    h(\vE_\alpha(\xi),u)
    =\frac{1}{\alpha}\int_{1-\alpha}^{1}q_t(\langle \xi,u\rangle)\dint t
  \end{displaymath}
  and
  \begin{displaymath}
    h(\vE_\alpha(\emm),u)
    =\frac{1}{\alpha}\int_{1-\alpha}^{1}q_t(\langle \hat\xi_n,u\rangle)\dint t\,.
  \end{displaymath}
  Clearly, $\langle \hat\xi_n,u\rangle$ is distributed according to
  the empirical distribution function generated by the sample $\langle
  \xi_i,u\rangle$, $i=1,\dots,n$.  Thus, the right-hand sides of the
  two equations are, respectively, the conditional value at risk of
  $\beta=\langle \xi,u\rangle$ and its sample-based estimator, see
  \cite{brow07, wan:gao10}. Note that the support of $\beta$ is a
  subset of an interval of length $R$. By \cite[Th.~3.1]{wan:gao10},
  for any $\eta>0$,
  \begin{displaymath}
    \Prob{h(\vE_\alpha(\emm),u)\le h(\vE_\alpha(\xi),u)-\eta}
    \leq 3 \exp\left\{-\frac{\alpha\eta^2 n}{5R^2}\right\}
  \end{displaymath}
  and
  \begin{displaymath}
    \Prob{h(\vE_\alpha(\emm),u)\ge h(\vE_\alpha(\xi),u)+\eta}
    \leq 3 \exp\left\{-\frac{\alpha\eta^2 n}{11R^2}\right\}\,.
  \end{displaymath}
  Noticing that the second bound is larger than the first one and that
  $h(\vE_\alpha(\xi),u)\geq r$ by the imposed condition, we obtain
  \begin{multline}
    \label{eq:double}
    \P\Big\{(1-\eps/2)h(\vE_\alpha(\xi),u)
      \le h(\vE_\alpha(\emm),u)\le (1+\eps/2)h(\vE_\alpha(\xi),u)\Big\}\\
    \ge 1-6\exp\left\{-\frac{\alpha\eps^2 r^2 n}{44R^2}\right\}\,.
  \end{multline}
  Let $\sN\subset\partial \vE_\alpha(\xi)^o$ be a set from
  Lemma~\ref{lem:net}, with $\delta=\frac{\eps}{2+2\eps}$, where
  $\vE_\alpha(\xi)^o$ is the polar set to $\vE_\alpha(\xi)$, see
  \eqref{eq:polar}.  Since $h(\vE_\alpha(\xi),u)=1$ for all
  $u\in\partial\vE_\alpha(\xi)^o$, the union bound applied to
  \eqref{eq:double} yields that
  \begin{equation}
    \label{eq:union}
    \left(1-\eps/2\right) \le h(\vE_\alpha(\emm),w) 
    \le \left(1+\eps/2\right)\quad \text{for all }\; w\in\sN
  \end{equation}
  with probability at least
  \begin{displaymath}
    1-6\left(\frac{6+6\eps}{\eps}\right)^d  
    \exp\left\{-\frac{\alpha\eps^2 r^2 n}{44R^2}\right\}.
  \end{displaymath}
  For any $v\in\partial\vE_\alpha(\xi)^o$ and some
  sequences $w_i\in\sN$ and $\delta_i\geq 0$, $i\geq1$, the
  sublinearity of $h$, Lemma~\ref{lem:net} and
  \eqref{eq:union} imply that
  \begin{align*}
    h(\vE_\alpha(\emm),v)
    &=h\left(\vE_\alpha(\emm),w_0+\sum_{i=1}^\infty \delta_i
      w_i\right)\\
    &\le \left(1+\eps/2\right)\sum_{i=0}^\infty
    \left(\frac{\eps}{2+2\eps}\right)^i\\
    &= \left(1+\eps/2\right)\frac{1}{1-\left(\frac{\eps}{2+2\eps}\right)}
    =(1+\eps)\,h(\vE_\alpha(\xi),v)
  \end{align*}
  and
  \begin{align*}
    h(\vE_\alpha(\emm),v)
    &=h\left(\vE_\alpha(\emm),w_0+\sum_{i=1}^\infty \delta_i
      w_i\right)\\
    &\ge \left(1-\eps/2\right)-\left(1+\eps/2\right)
    \sum_{i=1}^\infty\left(\frac{\eps}{2+2\eps}\right)^i\\
    &= \left(1-\eps/2\right)-\left(1+\eps/2\right)
    \frac{\left(\frac{\eps}{2+2\eps}\right)}{1-\left(\frac{\eps}{2+2\eps}\right)}
    = (1-\eps)\,h(\vE_\alpha(\xi),v)\,,
  \end{align*}
  which deliver the desired assertion. 
\end{proof}

\section{Floating-like bodies}
\label{sec:conv-bodi-gener-1}

\subsection{Sublinear transform}
\label{sec:sublinear-map}

Consider the set-valued sublinear expectation $\vE_\ve$ generated by a
law-determined numerical sublinear expectation $\ve$.  Let $\xi$ be a
random vector uniformly distributed on a convex body
$K\subset\R^d$. Recall that $K$ is assumed to have a nonempty
interior.  In the following, we write $\vE_\ve(K)$ instead of
$\vE_\ve(\xi)$ and refer to $K\mapsto\vE_\ve(K)$ as a \emph{sublinear
  transform} of $K$ generated by the numerical sublinear expectation
$\ve$. We also refer to $\vE_\ve(K)$ as a \emph{floating-like body}.

Denoting by $\sK$ the family of convex bodies in $\R^d$, the sublinear
transform is a map $\vE_\ve:\sK\to\sK$. It is easy to see that
$\vE_\ve(K)\subset K$ for all $K$. If $\xi$ is uniformly distributed on
$K$ and $A$ is a nondegenerate matrix, then $A\xi$ is uniformly
distributed on $AK$. Thus,
\begin{displaymath}
  \vE_\ve(AK+a)=A\vE_\ve(K)+a, \quad a\in\R^d.  
\end{displaymath}
The sublinear transform $\vE_\ve(B)$ of a centred Euclidean ball $B$ is
another centred Euclidean ball, which is contained in
$B$. Furthermore, the sublinear transform of an ellipsoid is also an
ellipsoid.

The sublinear transform is not necessarily monotone for inclusion, see
Example~\ref{ex:non-monotone}. In view of Remark~\ref{rem:order},
$\vE_\ve(K)\subseteq\vE_\ve(L)$ for all sublinear transforms $\vE_\ve$ if
\begin{displaymath}
  \frac{1}{V_d(K)}\int_K f(x)\dint x
  \leq \frac{1}{V_d(L)}\int_L f(x)\dint x
\end{displaymath}
for all convex functions $f:\R^d\to\R$, where $V_d(\cdot)$ denotes the
$d$-dimensional Lebesgue measure. The latter condition implies
that $K$ and $L$ share the same barycentre. 

If $K_n\to K$ in the Hausdorff metric as $n\to\infty$ and $\xi_n,\xi$
are uniformly distributed on $K_n,K$, respectively, then $\xi_n\to\xi$
in $\sigma(\Lp,\Lp[q])$ for any $p\in[1,\infty]$ by the dominated
convergence theorem. By Theorem~\ref{thr:subl-vE}(v),
$h(\vE_\ve(K),u)\leq \liminf h(\vE_\ve(K_n),u)$.

The continuity of the sublinear map in the Hausdorff metric follows
from the next result, which we find interesting in its own right. Denote by
$\diam(K)$ the diameter of $K$ and by $K\triangle L$ the symmetric
difference of $K$ and $L$.

\begin{theorem}
  \label{th:lip}
  Assume that $p\in [1,\infty)$. For any two convex bodies $K$ and
  $L$, there exist random vectors $\xi$ and $\eta$ uniformly
  distributed on $K$ and $L$, respectively, such that
  \begin{equation}
    \label{eq:bound2}
    \|\xi-\eta\|_p\le 
    \left(\frac{V_d(K\triangle L)}{\max(V_d(L),V_d(K))}\right)^\frac{1}{p}
    \diam(K\cup L)\,.
  \end{equation}
\end{theorem}
\begin{proof}
  It suffices to prove the statement for $p=1$. Indeed, 
  \begin{displaymath}
    \|\xi-\eta\|_{p}
    \le \diam(K\cup L)^{(p-1)/p}\|\xi-\eta\|_1^{1/p}. 
  \end{displaymath}
  Consider Monge's optimal transport problem of finding
  \begin{equation}
    \label{eq:Monge}
    \fC(\mu,\nu)=\inf_{T_\sharp\mu=\nu}\int_{\Rd}\|x-T(x)\| \dint\mu(x)\,,
  \end{equation}
  where $\mu$ and $\nu$ are the uniform distributions on $K$ and $L$,
  respectively, and $T_\sharp\mu$ denotes the push-forward of the
  measure $\mu$ by $T$.  It is known from the theory of optimal mass
  transportation (see, e.g., \cite{amb03} or \cite{vill03}) that the
  infimum in \eqref{eq:Monge} is attained on an optimal transport map
  $T$.  Moreover, under our assumptions, \cite[Th.~B]{pra07} yields
  the equivalence between Monge's transport problem and its
  alternative formulation by Kantorovich. Namely,
  \begin{displaymath}
    \fC(\mu,\nu)=\min_{\gamma\in\Pi(\mu,\nu)}
    \iint_{\Rd\times\Rd} \|x-y\| \dint\gamma(x,y)\,,
  \end{displaymath}
  where $\Pi(\mu,\nu)$ denotes the family of probability measures on
  $\Rd\times\Rd$ with marginals $\mu$ and $\nu$.  In other words,
  $\fC(\mu,\nu)$ is the $1$-Wasserstein distance between $\mu$ and
  $\nu$.  The dual representation of Kantorovich's problem
  (e.g. \cite[Th.~1.14]{vill03}) yields that 
  \begin{equation}
    \label{eq:dual}
    \min_{\gamma\in\Pi(\mu,\nu)}\iint_{\Rd\times\Rd} \|x-y\| \dint\gamma(x,y)
    =\max_{f\in\mathrm{Lip}_1} \left\{\int_{\Rd}f(x)\dint\mu (x)
      -\int_{\Rd}f(x)\dint\nu (x)\right\},
  \end{equation}
  where $\mathrm{Lip}_1$ is the family of $1$-Lipschitz functions on
  $\Rd$.  

  By adding a constant to $f$, one can restrict the
  maximisation in \eqref{eq:dual} to the set of $1$-Lipschitz
  functions with values in $[0,\diam(K\cup L)]$. Then
  \begin{align*}
    \int_{\Rd}f(x)\dint\mu (x)&-\int_{\Rd}f(x)\dint\nu (x)
    =\frac{1}{V_d(K)}\int_K f(x)\dint x-\frac{1}{V_d(L)}\int_L f(x)\dint x\\
    &=\frac{V_d(L)-V_d(K)}{V_d(K)V_d(L)}\int_{K\cap L}f(x)dx
    +\frac{1}{V_d(K)}\int_{K\setminus L}f(x)\dint x
    -\frac{1}{V_d(L)}\int_{L\setminus K}f(x)\dint x\\
    &\leq \left(
      \frac{V_d(L\setminus K)}{V_d(K)}\frac{V_d(K\cap L)}{V_d(L)}+
    \frac{V_d(K\setminus L)}{V_d(K)}\right) \diam(K\cup L)\\
    &\leq 
    \left(\frac{V_d(L\setminus K)}{V_d(K)}+\frac{V_d(K\setminus
      L)}{V_d(K)}\right) \diam(K\cup L)\\
    &= 
    \frac{V_d(K\triangle L)}{V_d(K)}\diam(K\cup L) \,. 
  \end{align*}
  Changing the order of summands, one obtains a similar bound with
  $V_d(K)$ replaced by $V_d(L)$, hence the result. 
\end{proof}

\begin{theorem}
  \label{th:cont}
  Let $\ve$ be a sublinear expectation defined on $\Lp(\R)$ for some
  $p\in [1,\infty)$ and having finite values. Then the map
  $K\mapsto\vE_\ve(K)$ is continuous in the Hausdorff metric.
\end{theorem}
\begin{proof}
  Note that the convergence of convex bodies (with nonempty
  interiors) in the Hausdorff metric is equivalent to their
  convergence in the symmetric difference metric, see
  \cite{shep:web65}. If $K_n\to K$ in the Hausdorff metric, then
  $\cup_n K_n$ is bounded and $\inf_n V_d(K_n)$ is strictly
  positive. By Theorem~\ref{th:lip}, it is possible to find a sequence
  of random vectors $\{\xi_n,n\geq1\}$ such that $\xi_n$ is uniformly
  distributed on $K_n$ and $\xi_n$ converges in $\Lp$ to a random
  vector $\xi$ uniformly distributed on $K$. The result follows from
  Theorem~\ref{thr:subl-vE}(vi). 
\end{proof}


\begin{example}
  The construction of the sublinear transform can be amended by
  replacing the underlying sublinear expectation $\ve$ with a (not
  necessarily convex) gauge function. For example, if the gauge
  function is a quantile, one obtains the set $D_\alpha(K)$, which is
  the \emph{convex floating body} of $K$, see \cite{bar:lar88} and
  \cite{schut:wer90}.
\end{example}

\subsection{Ulam floating bodies}
\label{sec:ulam-floating-bodies}

Consider the sublinear transform $K\mapsto\vE_\alpha(K)$ generated by
the average quantile sublinear expectation $\ve_\alpha$. Note that
$\vE_1(K)=\{x_K\}$ is the barycentre of $K$ (the expectation of $\xi$
uniformly distributed in $K$), and $\vE_0(K)=K$.

The metronoid $\Met(\mu)$ of the measure $\mu$ with density
$\delta^{-1}\one_K$ is called the \emph{Ulam floating body} of $K$ at
level $\delta$ and is denoted by $\Met_\delta(K)$, see
\cite{huan:slom:wer18}. This measure $\mu$ is the uniform probability
distribution on $K$ scaled by $\delta^{-1}V_d(K)$.
Proposition~\ref{prop:metro-averquant} yields that
\begin{equation}
  \label{eq:19}
  \vE_\alpha(K)=\Met_{\alpha V_d(K)}(K).
\end{equation}
Affine equivariance of sublinear transforms implies that
$M_\delta(cK)=cM_{\delta c^{-d}}(K)$.  Since the uniform probability
distribution on $K$ is log-concave, \eqref{eq:stat2} yields a
relationship between convex floating bodies of $K$ (denoted by
$D_\alpha(K)$) and Ulam floating bodies, proved in
\cite[Th.~1.1]{huan:slom:wer18}.

The following result for $\alpha\in(0,1/2)$ follows from
Theorem~\ref{th:bob}, see also \cite{mey:reis91}. Together with
\eqref{eq:19}, it implies that Ulam floating bodies can be obtained as
Aumann integrals of convex floating bodies. The case $\alpha=1/2$
follows by continuity.

\begin{corollary}
  \label{cor:integral}
  For each origin symmetric convex body $K$ and $\alpha\in(0,1/2]$, we
  have
  \begin{displaymath}
    \vE_\alpha(K)=\frac{1}{\alpha} \int_0^\alpha D_t(K)\dint t. 
  \end{displaymath}
\end{corollary}

Hence, $\alpha\vE_\alpha(K)$ grows in $\alpha$ for $\alpha\in(0,1/2]$,
equivalently, the dilated Ulam floating body $tM_t(K)$ grows for
$t\in(0,V_d(K)/2]$. 

The next result follows from Theorem~\ref{th:representation}; it
implies that Ulam floating bodies are building blocks for all
sublinear transforms.

\begin{corollary}
  \label{cor:ulam}
  For each law-determined sublinear expectation $\ve$, the
  corresponding sublinear transform $\vE_\ve$ can be represented as
  \begin{equation}
    \label{eq:7}
    \vE_\ve(K)=\conv \bigcup_{\nu\in\sP_\ve} \int_{(0,1]}
    \vE_\alpha(K)\nu(\dint\alpha),
  \end{equation} 
  where $\nu$ runs through a family $\sP_\ve$ of probability measures
  on $(0,1]$ that yields the Kusuoka representation of $\ve$, see
  \eqref{eq:12}.
\end{corollary}

It is possible to replace $\vE_\alpha$ with $M_{\alpha V_d(K)}$ on the
right-hand side of \eqref{eq:7}. While the
integration domain in \eqref{eq:7} excludes $0$, it is always possible
to approximate $\vE_0(K)=K$ by a sequence $\vE_{\alpha_n}(K)$ as
$\alpha_n\downarrow0$. Thus, the Kusuoka representation can be
equivalently written using probability measures on $[0,1]$.

\begin{example}
  \label{ex:non-monotone}
  The map $K\mapsto\vE_\alpha(K)$ is not necessarily monotone. An easy
  counterexample is provided by two segments $[0,1]$ and $[0,2]$ on
  the line. However, the monotonicity fails even for origin symmetric
  convex bodies. Consider two convex bodies on the plane:
  $L=[-a,a]\times[-\eps,\eps]$ with $a+\eps\leq 1$ and the
  $\ell_1$-ball $K$. We show that for suitable values of $a$ and
  $\alpha$, the support function of $\vE_\alpha(L)$ is not smaller
  than the support function of $\vE_\alpha(K)$ in direction
  $u=(1,0)$. Let $\beta=\langle\xi,u\rangle$ for $\xi$ uniformly
  distributed in $K$. Note that $\gamma=\langle \eta,u\rangle$ is
  uniformly distributed on $[-a,a]$ if $\eta$ is uniform on $L$. The
  quantile functions are
  \begin{displaymath}
    q_t(\beta)=1-\sqrt{2(1-t)},\quad q_t(\gamma)=(2t-1)a,\quad t\in[1/2,1].
  \end{displaymath}
  For $\alpha\in[0,1/2]$,
  \begin{displaymath}
    \ve_\alpha(\beta)=1-2\sqrt{2}\alpha^{1/2}/3
  \end{displaymath}
  and 
  \begin{displaymath}
    \ve_\alpha(\gamma)=a(1-\alpha).
  \end{displaymath}
  If $\alpha=1/2$, then $\ve_\alpha(\beta)<\ve_\alpha(\gamma)$ if
  $\frac{2}{3}<a<1$,
  meaning that $\vE_\alpha(L)$ is not necessarily a subset of
  $\vE_\alpha(K)$.

  The monotonicity of Ulam floating body transform (which easily
  follows from Proposition~2.1 of \cite{huan:slom:wer18}) implies that,
  after normalising by volume, $\vE_\alpha$ becomes monotone, namely,
  \begin{displaymath}
    \vE_{\alpha/V_d(K)}(K)\subset \vE_{\alpha/V_d(L)}(L), \quad
    0\leq\alpha\leq V_d(K),
  \end{displaymath}
  if $K\subset L$. 
\end{example}

If the family $\sP_\ve$ in \eqref{eq:7} consists of a single measure
$\nu$, we obtain a convex body $\vE_{\int\!\!\phi}(K)$ generated by
the spectral sublinear expectation $\ve_{\int\!\!\phi}$, where $\phi$
is the spectral function related to $\nu$ by~\eqref{eq:13}. Recall
that the maximum extension of the average quantile is a spectral
sublinear expectation, see Example~\ref{ex:max-ext}.

\begin{example}
  \label{ex:exp-poly}
  Consider the sublinear expectation $\ve^{\vee m}_1$ given by
  \eqref{eq:max}.  Note that
  \begin{displaymath}
    \max(\langle u,\xi_1\rangle,\dots,\langle u,\xi_m\rangle)
    =h(P_m,u),
  \end{displaymath}
  where $P_m=\conv(\xi_1,\dots,\xi_m)$ is the convex hull of
  independent copies of $\xi$.  Then $\E h(P_m,u)$ is the support
  function of the expectation $\E P_m$ of the random polytope $P_m$,
  see \cite[Sec.~2.1]{mo1}. Therefore, $\vE^{\vee m}_1(K)=\E
  P_m$. Asymptotic properties of these expected polytopes and their
  relation to floating bodies have been studied in \cite{fres:vit14},
  see also \cite{fres13}.  If $m=1$, then $\vE_1(K)=\{x_K\}$ is the
  barycentre of $K$. The calculation  in Example~\ref{ex:max-ext}
  yields that
  \begin{align*}
    \E P_m=\vE_1^{\vee m}(K)&=m(m-1)\int_{(0,1]}\vE_\alpha(K)
            \alpha(1-\alpha)^{m-2}\dint \alpha\\
    &=m(m-1)\int_{(0,1]}\Met_{\alpha V_d(K)}(K)
            \alpha(1-\alpha)^{m-2}\dint \alpha.
  \end{align*}
  Hence, the expected random polytope equals the weighted integral of
  Ulam floating bodies. 
\end{example}

More generally, $\vE_\alpha^{\vee m}(K)$ is obtained by applying
\eqref{eq:alpha-m} as follows
\begin{multline*}
  \vE_\alpha^{\vee m}(K)
  =\frac{m(m-1)}{\alpha}
  \int _0^{1-(1-\alpha)^{1/m}} t(1-t)^{m-2}\vE_t(K)\dint t\\
  +\frac{m}{\alpha}(1-\alpha)^{(m-1)/m}(1-(1-\alpha)^{1/m})
  \vE_{1-(1-\alpha)^{1/m}}(K).
\end{multline*}

\subsection{Centroid bodies and the expectile transform}
\label{sec:centroid-bodies}

If $\ve_{p,a}$ is defined by \eqref{eq:norm} for $p\in[1,\infty)$,
then the corresponding floating-like body $\vE_{p,a}(K)$ has the
support function
\begin{equation}
  \label{eq:15}
  h(\vE_{p,a}(K),u)=\langle x_K,u\rangle
  + a\Big(\E(\langle\xi-x_K,u\rangle)_+^p\Big)^{1/p},
\end{equation}
where $\xi$ is uniformly distributed on $K$ and $x_K=\E\xi$ is the
barycentre of $K$.

If $K$ is origin-symmetric, then $x_K=0$ and
\begin{displaymath}
  \vE_{p,1}(K)=c \Gamma_pK,  
\end{displaymath}
where $c>0$ is an explicit constant depending on $p$ and dimension and
$\Gamma_pK$ is the \emph{$\Lp$-centroid body} of $K$, see \cite{lut90}
for $p=1$ and \cite{lut:zhan97} for general $p$. For a not necessarily
origin symmetric $K$, this convex body is defined as
\begin{displaymath}
  h(\Gamma_pK,u)=\left(\frac{1}{c_{d,p}V_d(K)}
    \int_K |\langle u,y\rangle|^p \dint y\right)^{1/p},
\end{displaymath}
where $c_{d,p}$ is a constant chosen to ensure that this
transformation does not change the unit Euclidean ball, see
\cite[Eq.~(10.72)]{schn2}.
For $p=1$, $a=1$ and an origin symmetric $K$,
\begin{displaymath}
  \vE_{1,1}(K)=\frac{1}{2} \Gamma K,
\end{displaymath}
where $\Gamma K$ is the classical \emph{centroid body} of $K$, see
\cite[Eq.~(10.67)]{schn2} and \cite{lut90}.  The dual representation
of $\ve_{1,1}$ from Example~\ref{ex:fischer} yields that
\begin{displaymath}
  \Gamma K=2\vE_{1,1}(K)=\conv \{\E(\gamma\xi):\; \gamma\in[0,2]\}.
\end{displaymath}
The right-hand side is the expectation of the random convex body
$[0,2\xi]$ being the segment in $\R^d$ with end-points at the origin
and $2\xi$, see \cite[Sec.~2.1]{mo1}. 

The \emph{asymmetric $\Lp$-moment body} $\Met_p^+K$ introduced in
\cite{hab:sch09} (see also \cite[Eq.~(10.76)]{schn2}) has the support
function proportional to
\begin{displaymath}
  \int_K (\langle x,y\rangle)_+^p \dint y. 
\end{displaymath}
Thus,
\begin{displaymath}
  \vE_{p,a}(K)=x_K+c_1a\Met_p^+(K-x_K)
\end{displaymath}
for a constant $c_1$ depending on $p\in[1,\infty)$ and
dimension.

Corollary~\ref{cor:dual-E} and the dual representation of $\ve_{p,a}$
from Example~\ref{ex:fischer} (see also \cite[p.~46]{del12}) yield
that the asymmetric $\Lp$-moment bodies with $p\in[1,\infty)$ can be
represented in terms of 
\begin{displaymath}
  \vE_{p,a}(K)=x_K+a \cl\big\{\E((\gamma-\E\gamma)\xi):\; 
  \gamma\in\Lp[q](\R_+), \|\gamma\|_q\leq 1\big\}. 
\end{displaymath}

Furthermore, Corollary~\ref{cor:ulam} shows that each $\Lp$-centroid
body of an origin symmetric $K$ equals the convex hull of a family of
integrated Ulam floating bodies of $K$. This representation can be
made very explicit in case $p=1$; it follows from
Theorem~\ref{thr:law-inv-rep} combined with the results presented in
Example~\ref{ex:fischer}. Namely,
\begin{equation}
  \label{eq:18}
  \vE_{1,a}(K)=x_K+ a\conv\bigcup_{t\in[0,1]} t\vE_t(K-x_K).
\end{equation}
The following result specialises the above
relationship for centroid bodies.

\begin{corollary}
  \label{cor:centroid}
  If $K$ is an origin symmetric convex body, then its centroid body
  $\Gamma K$ satisfies
  \begin{equation}
    \label{eq:16}
    \Gamma K=\frac{2}{V_d(K)} \conv\bigcup_{t\in[0,V_d(K)]} t
    \Met_{t}(K).
  \end{equation}
\end{corollary}

Since $K$ is origin symmetric, $t\vE_t(K)=\int_0^t D_s(K)\dint s$
grows in $t\in(0,1/2]$, see Corollary~\ref{cor:integral}. Thus, the
union in \eqref{eq:16} can be reduced to $t\in[V_d(K)/2,1]$.

\begin{example}
  The definition of the Orlicz centroid bodies from
  \cite{lut:yan:zhan10} can be also incorporated in our setting using
  the sublinear expectation 
  \begin{displaymath}
    \ve(\beta)=\inf\{\lambda>0:\; \E \psi(\beta/\lambda)\leq 1\},
  \end{displaymath}
  where $\psi:\R\to[0,\infty)$ is a convex function with $\psi(0)=0$
  and such that $\psi$ is strictly increasing on the positive
  half-line or strictly decreasing on the negative half-line. This
  sublinear expectation is the norm of $\beta$ in the corresponding
  Orlicz space.
\end{example}

\begin{example}
  Consider the expectile $\ve_{[\tau]}$ defined in
  Example~\ref{ex:expectile} with parameter $\tau\in(0,1/2]$.
  In view of the results presented in Example~\ref{ex:expectile}, the
  corresponding floating-like body $\vE_{[\tau]}(K)$ can be
  represented as
  \begin{equation}
    \label{eq:17}
    \vE_{[\tau]}(K)
    =x_K+\conv\bigcup_{t\in[0,V_d(K)]}
    \frac{t(2\tau-1)}{t(2\tau-1)+(1-\tau)V_d(K)}
    \big(\Met_t(K)-x_K\big).
  \end{equation}
\end{example}

Representations \eqref{eq:16} and \eqref{eq:17} suggest looking at the
transform of convex bodies given by
\begin{displaymath}
  K\mapsto x_K+\conv\bigcup_{t\in[0,V_d(K)]} \psi(t)\big(\Met_t(K)-x_K\big)
\end{displaymath}
for a function $\psi:\R_+\to\R_+$. As demonstrated above, this
transform relates the centroid body transform and the expectile
transform to the Ulam floating body transform.

\subsection{Open problems related to the sublinear transform}
\label{sec:open-problems}

Several calculated examples suggest that
$\vE_\alpha(K+L)\subset \vE_\alpha(K)+\vE_\alpha(L)$, and we
conjecture that this is the case. It is easy to see that this holds on
the line for a general sublinear transform.

It was shown in \cite{guo:len:lin18} that the equality of two
symmetric $p$-centroid bodies for $p$ not being an even integer yields
the equality of the corresponding sets. This question is open for Ulam
floating bodies, see \cite{huan:slom:wer18}, not to say also for 
general floating-like bodies.

It is obvious that $\vE_\ve(K)$ is a dilate of $K$ if $K$ is an
ellipsoid. This question has been explored for convex floating bodies,
see \cite{MR2952407} and references therein. However, the case of Ulam
floating body seems to be open, as well as the case of general
sublinear transform. 

There is a substantial theory of conditional (dynamic) sublinear
expectations, e.g., constructed using backwards stochastic
differential equations, see \cite{pen19}. By applying conditional
sublinear expectations to $\xi$ uniformly distributed in $K$, one
comes up with stochastic processes whose values are convex
bodies. Further investigation of such processes is left for future
work.

\section*{Acknowledgement}
\label{sec:acknowledgement}

The authors are grateful to the referee for a very careful reading of
the paper and suggesting numerous improvements.


\end{document}